\documentclass[11pt]{amsart}
\usepackage{amsmath,amssymb}
\usepackage{amscd,eucal,amsthm}

\usepackage{hyperref}

\newtheorem{Theorem}{Theorem}

\newtheorem{Lemma}{Lemma}
\newtheorem{Fact}{Fact}
\newtheorem{Proposition}{Proposition}

\theoremstyle{remark}
\newtheorem*{Remark}{Remark}

\renewcommand{\to}[1][]{\xrightarrow{#1}}

\renewcommand{\gg}{{\mathfrak{g}}}
\renewcommand{\ll}{{\mathfrak{l}}}

\newcommand{\z}{{\mathfrak{z}}}

\newcommand{\hh}{{\mathfrak{h}}}
\renewcommand{\l}{{\lambda}}

\newcommand{\zz}{\mathbb{Z}}

\newcommand{\cc}{\mathbb{C}}

\newcommand{\A}{\mathcal{A}}
 
\DeclareMathOperator{\gr}{gr} 
 
 \DeclareMathOperator{\id}{id}

\DeclareMathOperator{\End}{End} 
\DeclareMathOperator{\Res}{Res} 
\DeclareMathOperator{\Spec}{Spec}
 \DeclareMathOperator{\rk}{rk}
\DeclareMathOperator{\Tr}{Tr} \DeclareMathOperator{\diag}{diag}
\renewcommand{\phi}{\varphi}

\newcommand{\ep}{\varepsilon}

\makeatletter
\def\@mult#1{\raise #1\rlap{$\cdot$}\lower #1\rlap{$\cdot$}\cdot}
\def\did{\mathrel{\@mult{3pt}}}
\makeatother
\def\openrow#1#2#3{\setbox0=\vbox{\hbox
    {\vrule height#2 width#3\kern#2\vrule height#2 width0pt}\hrule height#3}
    \hbox{\leaders\copy0\hskip#1\wd0\vrule width#3}}
\def\row#1#2#3{\vbox{\hrule height#3\openrow{#1}{#2}{#3}}}
\def\Yr#1{\row{#1}{1.5ex}{.1ex}}

\def\DY#1\endDY{\baselineskip=1ex\lineskip=0pt\lineskiplimit=0pt{\vcenter
    {\Yr#1}}}
\def\openclm#1#2#3{\setbox0=\vbox{\hrule height#3\hbox
    {\vrule width0pt\kern#2\vrule width#3 height#2}}\vtop
    {\leaders\copy0\vskip#1\ht0\hrule height#3}}
\def\clm#1#2#3{\hbox{\vrule width#3\openclm{#1}{#2}{#3}}}
\def\Yc#1{\clm{#1}{1.5ex}{.1ex}}

\def\CDY#1\endCDY{{\vcenter{\hbox{\Yc#1}}}}

\def\matrixSize{{r}}

\author{A.~Chervov}
\address{Institute for Theoretical and Experimental
Physics, Moscow} \email{chervov@itep.ru}

\author{G.~Falqui}
\address{Universita degli Studi di Milano - Bicocca}  \email{E-mail:gregorio.falqui@unimib.it}

\author{L.~Rybnikov}
\address{Poncelet laboratory (Independent University of Moscow and CNRS) and Institute for Theoretical and Experimental
Physics, Moscow} \email{leo.rybnikov@gmail.com}

\title{ \hfill {\em ITEP-TH-78/06} \\ ~~\\
Limits of Gaudin algebras, quantization of bending flows, \\ ~~ \\ Jucys--Murphy elements and Gelfand--Tsetlin bases}

\begin{document}
\maketitle

\begin{abstract}
Gaudin algebras form a family of maximal commutative subalgebras
in the tensor product of $n$ copies of the universal enveloping
algebra $U(\gg)$ of a semisimple Lie algebra $\gg$. This family is
parameterized by collections of pairwise distinct complex numbers
$z_1,\dots,z_n$. We obtain some new commutative subalgebras in
$U(\gg)^{\otimes n}$ as limit cases of Gaudin subalgebras. These
commutative subalgebras turn to be related to the hamiltonians of
bending flows and to the Gelfand--Tsetlin bases. We use this to
prove the simplicity of spectrum in the Gaudin model for some new
cases.
\end{abstract}

\bigskip


\tableofcontents

\section{Introduction}

Gaudin model was introduced in \cite{G1} as a spin model related
to the Lie algebra $sl_2$, and generalized to the case of
arbitrary semisimple Lie algebras in \cite{G}, 13.2.2. The
generalized Gaudin model has the following algebraic
interpretation. Let $V_{\lambda}$ be an irreducible representation
of $\gg$ with the highest weight $\lambda$. For any collection of
integral dominant weights $(\lambda)=\lambda_1,\dots,\lambda_n$,
let $V_{(\l)}=V_{\l_1}\otimes\dots\otimes V_{\l_n}$. For any
$x\in\gg$, consider the operator $x^{(i)}=1\otimes\dots\otimes
1\otimes x\otimes 1\otimes\dots\otimes 1$ ($x$ stands at the $i$th
place), acting on the space $V_{(\l)}$. Let $\{x_a\},\
a=1,\dots,\dim\gg$, be an orthonormal basis of $\gg$ with respect
to Killing form, and let $z_1,\dots,z_n$ be pairwise distinct
complex numbers. The hamiltonians of Gaudin model are the
following commuting operators acting in the space $V_{(\l)}$:
\begin{equation}\label{quadratic}
H_i=\sum\limits_{k\neq i}\sum\limits_{a=1}^{\dim\gg}
\frac{x_a^{(i)}x_a^{(k)}}{z_i-z_k}.
\end{equation}

We can treat the $H_i$ as elements of the universal enveloping
algebra $U(\gg)^{\otimes n}$. In \cite{FFR}, an existence of  a large commutative
subalgebra $\A(z_1,\dots,z_n)\subset U(\gg)^{\otimes n}$
containing $H_i$ was proved. 
For $\gg=sl_2$, the algebra
$\A(z_1,\dots,z_n)$ is generated by $H_i$ and the central elements
of $U(\gg)^{\otimes n}$. In other cases, the algebra
$\A(z_1,\dots,z_n)$ has also some new generators known as higher
Gaudin hamiltonians. Their explicit construction
for $\gg=gl_\matrixSize$ was obtained in \cite{T04}, see also \cite{CT04},\cite{CT06}.
(For $\gg=gl_3$ see \cite{CRT}).
The construction of $\A(z_1,\dots,z_n)$ uses
the quite nontrivial fact \cite{FF} that the completed universal
enveloping algebra of the affine Kac--Moody algebra $\hat{\gg}$ at
the critical level has a large center $Z(\hat{\gg})$. There is a
natural homomorphism from the center $Z(\hat{\gg})$ to the
enveloping algebra $U(\gg\otimes t^{-1}\cc[t^{-1}])$. To any
collection $z_1,\dots,z_n$ of pairwise distinct complex numbers,
one can naturally assign the evaluation homomorphism $U(\gg\otimes
t^{-1}\cc[t^{-1}])\to U(\gg)^{\otimes n}$. The image of the center
under the composition of the above homomorphisms is
$\A(z_1,\dots,z_n)$. We will call $\A(z_1,\dots,z_n)$ the
\emph{Gaudin algebra}.

The main problem in Gaudin model is the problem of simultaneous
diagonalization of (higher) Gaudin hamiltonians. The bibliography
on this problem is enormous (cf.
\cite{BF,Fr1,Fr2,FFR,FFTL,MV,MTV,MTV2}). It follows from the
\cite{FFR} construction that all elements of
$\A(z_1,\dots,z_n)\subset U(\gg)^{\otimes n}$ are invariant with
respect to the diagonal action of $\gg$, and therefore it is
sufficient to diagonalize the algebra $\A(z_1,\dots,z_n)$ in the
subspace $V_{(\l)}^{sing}\subset V_{(\l)}$ of singular vectors
with respect to $\diag_n(\gg)$ (i.e., with respect to the diagonal
action of $\gg$). The standard conjecture says that, for generic
$z_i$, the algebra $\A(z_1,\dots,z_n)$ has simple spectrum in
$V_{(\l)}^{sing}$. This conjecture is proved in \cite{MV} for
$\gg=sl_r$ and $\l_i$ equal to $\omega_1$ or $\omega_{r-1}$ (i.e.,
for the case when every $V_{\l_i}$ is the standard representation
of $sl_r$ or its dual) and in \cite{SV} for $\gg=sl_2$ and
arbitrary $\l_i$.

In the present paper we consider some limits of the Gaudin
algebras when some of the points $z_1,\dots,z_n$ glue together. We
obtain some new commutative subalgebras this way. We consider the
"most degenerate" subalgebra of this type. In the case of
$\gg=gl_N$, this subalgebra gives a quantization of the (higher)
"bending flows hamiltonians", introduced in \cite{FM03-1}.
Original bending flows were introduced in \cite{KM96}. They are related to $SU(2)$.
See \cite{Rag,FlashM01,FM03-2,FM04,FM06} for further developments.
Further, we
establish a connection of this subalgebra to the Gelfand--Tsetlin
bases via the results of Mukhin, Tarasov and Varchenko on
$(gl_N,gl_M)$ duality (cf \cite{MTV,MTV2}). This result was obtained first
in \cite{FlashM01} by different methods. We use this to prove
the simple spectrum conjecture for $\gg=gl_N$ and
$\l_i=m_i\omega_1,\ m_i\in\zz_+$.

The paper is organized as follows. In section~ \ref{sect_prelim} we
collect some well-known facts on Gaudin algebras. In
section~\ref{sect_limit} we describe some limits of Gaudin
algebras. In section~ \ref{sect_bendflows} we obtain quantum
hamiltonians of bending flows as limits of higher Gaudin
hamiltonians. In sections~ \ref{sect_jucis}~and~ \ref{sect_duality}
we establish a connection between quantum bending flows
hamiltonians and Gelfand--Tsetlin theory. Finally, in
section~ \ref{sect_simplespec} we apply our results and prove the
simple spectrum conjecture.

{\bf Acknowledgments} G.F. acknowledges support from the ESF
programme MISGAM, and the Marie Curie RTN ENIGMA. The work of A.C.
and L.R. has been partially supported by the RFBR grant
04-01-00702 and by the Federal agency for atomic energy of Russia.
A.C. has been partially supported by the Russian President Grant MK-5056.2007.1,
 grant of Support for the Scientific Schools 8004.2006.2,
and by the INTAS grant YSF-04-83-3396, by the ANR grant GIMP
(Geometry and Integrability in Mathematics and Physics), the part
of work  was done during the visits to SISSA (under the INTAS
project), and to the University of Angers (under ANR grant GIMP).
A.C. is deeply indebted to SISSA and especially to B. Dubrovin, as
well as to the University of Angers and especially to V. Rubtsov,
for providing warm hospitality, excellent working conditions and
stimulating  discussions. The work of L.R. was partially supported
by  RFBR grant 05 01 00988-a and RFBR grant 05-01-02805-CNRSL-a.
L.R. gratefully acknowledges the support from Deligne 2004 Balzan
prize in mathematics. The work was finished during L.R.'s stay at
the Institute for Advanced Study supported by the NSF grant
DMS-0635607. The authors are indebted to A. Vershik for the
interest in their work, to A. Mironov for useful discussion.

\section{Preliminaries}\label{sect_prelim}

\subsection{Construction of Gaudin subalgebras.} Consider the infinite-dimensional pro-nilpotent Lie
algebra $\gg_-:=\gg\otimes t^{-1}\cc[t^{-1}]$ -- it is a "half" of
the corresponding affine Kac--Moody algebra $\hat\gg$. The
universal enveloping algebra $U(\gg_-)$ bears a natural filtration
by the degree with respect to the generators. The associated
graded algebra is the symmetric algebra $S(\gg_-)$ by the
Poincar\'e--Birkhoff--Witt theorem. The commutator operation on
$U(\gg_-)$  defines the Poisson--Lie bracket $\{\cdot,\cdot\}$ on
$S(\gg_-)$: for the generators $x,y\in\gg_-$ we have
$\{x,y\}=[x,y]$. For any $g\in\gg$, we denote the element
$g\otimes t^{m}\in\gg_-$ by $g[m]$.

The Poisson algebra $S(\gg_-)$ contains a large
Poisson-commutative subalgebra $A\subset S(\gg_-)$. This
subalgebra can be constructed as follows.

Consider the following derivations of the Lie algebra $\gg_-$:
\begin{equation}\label{der1}
\partial_t(g[m])=mg\otimes t^{m-1}\quad\forall g\in\gg, m=-1,-2,\dots
\end{equation}
\begin{equation}\label{der2}
t\partial_t(g[m])=mg\otimes t^{m}\quad\forall g\in\gg,
m=-1,-2,\dots
\end{equation}
The derivations (\ref{der1}), (\ref{der2}) extend to the
derivations of the associative algebras $S(\gg_-)$ and $U(\gg_-)$.
The derivation (\ref{der2}) induce a grading of these algebras.

Let $i_{-1}:S(\gg)\hookrightarrow S(\gg_-)$ be the embedding,
which maps $g\in\gg$ to $g[-1]$. Let $\Phi_l,\ l=1,\dots,\rk\gg$
be the generators of the algebra of invariants $S(\gg)^{\gg}$.

\begin{Fact} \begin{enumerate} \item \cite{BD,FFR,Fr2} The subalgebra  $A\subset S(\gg_-)$ generated by the elements
$\partial_t^n \overline{S_l}$, $l=1,\dots,\rk\gg$,
$n=0,1,2,\dots$, where $\overline{S_l}=i_{-1}(\Phi_l)$, is
Poisson-commutative. \item There exist the homogeneous with
respect to $t\partial_t$ elements $S_l\in\A$ such that $\gr
S_l=\overline{S_l}$. \item $\A$ is a free commutative algebra
generated by $\partial_t^n S_l$, $k=1,\dots,\rk\gg$,
$n=0,1,2,\dots$.\end{enumerate}
\end{Fact}

\begin{Remark} The generators of the subalgebra $A\subset S(\gg_-)$ can be described in the following
equivalent way. Let $i(z):S(\gg)\hookrightarrow S(\gg_-)$ be
the embedding depending on the formal parameter $z$, which maps
$g\in\gg$ to $\sum\limits_{k=1}^{\infty}z^{k-1}g[-k]$. Then the
coefficients of the power series $\overline{S_l}(z)=i(z)(\Phi_l)$
in $z$ freely generate the subalgebra $A\subset S(\gg_-)$.
\end{Remark}

\begin{Remark} The subalgebra $\A\subset U(\gg_-)$ comes from the center of
$U(\hat\gg)$ at the critical level by the AKS-scheme (see
\cite{FFR,ER,CT06} for details).
\end{Remark}

The Gaudin subalgebra $\A(z_1,\dots,z_n)\subset U(\gg)^{\otimes
n}$ is the image of the subalgebra $\A\subset U(\gg_-)$ under the
homomorphism $U(\gg_-)\to U(\gg)^{\otimes n}$ of specialization at
the points $z_1,\dots,z_n$ (see \cite{FFR,ER}). Namely, let
$U(\gg)^{\otimes n}$ be the tensor product of $n$ copies of
$U(\gg)$. We denote the subspace $1\otimes\dots\otimes
1\otimes\gg\otimes 1\otimes\dots\otimes 1\subset U(\gg)^{\otimes
n}$, where $\gg$ stands at the $i$th place, by $\gg^{(i)}$.
Respectively, for any $f\in U(\gg)$ we set
\begin{equation}
f^{(i)}=1\otimes\dots\otimes 1\otimes f\otimes
1\otimes\dots\otimes 1\in U(\gg)^{\otimes n}.
\end{equation}

Let $\diag_n:U(\gg_-)\hookrightarrow U(\gg_-)^{\otimes n}$ be the
diagonal embedding. For any collection of pairwise distinct
complex numbers $z_i, i=1,\dots,n$, we have the following
homomorphism:
\begin{equation}
\phi_{z_1,\dots,z_n}=(\phi_{z_1}\otimes\dots\otimes\phi_{z_n})\circ
\diag_n:U(\gg_-)\to U(\gg)^{\otimes n}.
\end{equation}

More explicitly, we have
$$\phi_{z_1,\dots,z_n}(g\otimes
t^m)=\sum\limits_{i=1}^nz_i^mg^{(i)}.$$

Set
$$
\A(z_1,\dots,z_n)=\phi_{z_1,\dots,z_n}(\A)\subset U(\gg)^{\otimes
n}
$$

\subsection{Quantum Mishchenko-Fomenko "shift of argument" subalgebras.}

Consider the subalgebra $\A(z_1,z_2)\subset U(\gg)\otimes U(\gg)$.
The associated graded quotient of $\A(z_1,z_2)$ with respect to
the second tensor factor is a commutative subalgebra
$\overline{\A(z_1,z_2)}\subset U(\gg)\otimes S(\gg)$. Any element
$\mu\in\gg^*=\Spec S(\gg)$ gives the evaluation homomorphism
$\id\otimes\mu:U(\gg)^{\otimes(n-1)}\otimes S(\gg)$. The image of
$\overline{\A(z_1,z_2)}$ under this homomorphism a commutative
subalgebra $\A_\mu\subset U(\gg)$, which does not depend on
$z_1,z_2$. This subalgebra is a quantum version of the
Mishchenko--Fomenko "shift of argument" subalgebra $A_\mu\subset
S(\gg)$ (see \cite{CT06,FFTL,MTV06,Ryb1}). The latter generated by
the derivatives (of any order) along $\mu$ of the generators of
the Poisson center $S(\gg)^{\gg}$, (or, equivalently, generated by
central elements of $S(\gg)=\cc[\gg^*]$ shifted by $t\mu$ for all
$t\in\cc$) \cite{MF}. In \cite{FFTL,Ryb1} it is shown, that
$\gr\A_\mu=A_\mu$. The algebra $A_\mu$ is a free commutative
algebra with $\frac{1}{2}(\dim\gg+\rk\gg)$ generators, and hence
has the maximal possible transcendence degree (see \cite{MF}).

\subsection{Talalaev's formula.} In~\cite{T04} D.~Talalaev constructed explicitly some
elements of $U(\gg)^{\otimes n}$ commuting with quadratic Gaudin
hamiltonians for the case $\gg=\gg\ll_r$. The formulas
of~\cite{T04} are universal, i.e. actually they describe a
commutative subalgebra of $U(\gg_-)$ which gives a commutative
subalgebra of $U(\gg)^{\otimes n}$ as the image of the
specialization homomorphism at the points $z_1,\dots,z_n$ (see
\cite{CT06}).

Namely, set
$$L(z)=\sum\limits_{1\le i,j\le
r}\sum\limits_{n=1}^{\infty}z^{n-1}e_{ij}[-n]\otimes e_{ji}\in
U(\gg_-)\otimes\End\cc^r,$$ where $z$ is a formal parameter, and
consider the following differential operator in $z$ with the
coefficients from $U(\gg_-)$:
$$
D=\Tr A_r\prod\limits_{i=1}^r(L(z)^{(i)}-\partial_z)=
\partial_z^r+\sum\limits_{k=1}^{r}\sum\limits_{n=1}^{\infty}Q_{n,k}z^{n-1
}\partial_z^{r-k}.\footnote{It was first observed in
\cite{MTVShap}, that one can use simply the column determinant to
define $D$, $D=det^{col} (L(z)-\partial_z)$, where $det^{col}
M=\sum_{\sigma} (-1)^{sigma} \prod M_{\sigma(i),i}$. This becomes
natural due to an observation of \cite{CF}, that $L(z)-\partial_z$
is a ``Manin's matrix'', so use of column-determinant is dictated
by Manin's general theory.}
$$

Here we denote by $L(z)^{(i)}\in
U(\gg_-)\otimes(\End\cc^r)^{\otimes r}$ the element obtained by
putting $L(z)$ to the $i$-th tensor factor, and $A_r$ denotes the
projector onto $U(\gg_-)\otimes\End(\Lambda^r\cc^r)\subset
U(\gg_-)\otimes(\End\cc^r)^{\otimes r}$. It follows from
\cite{CT06} that the elements $Q_{n,k}\in U(\gg_-)$ pairwise
commute.

In \cite{Ryb}, it is shown that
$\sum\limits_{n=1}^{\infty}Q_{n,k}z^{n-1 
}=S_l(z)$, and hence the
elements $Q_{n,k}\in U(\gg_-)$ generate the same commutative
subalgebra $\A\subset U(\gg_-)$.

\subsection{The generators of
$\A(z_1,\dots,z_n)$.}\label{subsect_generators} For our purposes,
we need some specific set of generators of $\A(z_1,\dots,z_n)$.
Let us describe it.

Consider the following $U(\gg)^{\otimes n}$-valued functions in
the variable $w$
$$
S_l(w;z_1,\dots,z_n):=\phi_{w-z_1,\dots,w-z_n}(S_l)
$$
Let $S_l^{i,m}(z_1,\dots,z_n)$ be the coefficients of the
principal parts of the Laurent series of $S_l(w;z_1,\dots,z_n)$ at
the points $z_1,\dots,z_n$. I.e. we set
$$S_l(w;z_1,\dots,z_n)=\sum\limits_{m=1}^{m=\deg\Phi_l}S_l^{i,m}(z_1,\dots,z_n)(z-z_i)^{-m}+O(1)\
\text{as}\ z\to z_i.$$

The following assertion is standard.

\begin{Proposition}\label{generators2}
\begin{enumerate}
\item The subalgebra $\A(z_1,\dots,z_n)$ is a free commutative
algebra generated by the elements $S_l^{i,m}(z_1,\dots,z_n)\in
U(\gg)^{\otimes n}$, where $i=1,\dots,n-1$, $l=1,\dots,\rk\gg$,
$m=1,\dots,\deg\Phi_l$, and
$S_l^{n,\deg\Phi_l}(z_1,\dots,z_n)=\Phi_l^{(n)}\in U(\gg)^{\otimes
n}$, where $l=1,\dots,\rk\gg$. \item The elements
$S_l^{i,m}(z_1,\dots,z_n)\in U(\gg)^{\otimes n}$ are stable under
simultaneous affine transformations of the parameters $z_i\mapsto
az_i+b$. \item All the elements of $\A(z_1,\dots,z_n)$ are
invariant with respect to the diagonal action of $\gg$. \item The
center of the diagonal $\diag_n(U(\gg))\subset U(\gg)^{\otimes n}$
is contained in $\A(z_1,\dots,z_n)$.
\end{enumerate}
\end{Proposition}

\begin{proof}
(1) The algebra $\A(z_1,\dots,z_n)$ is generated by the elements
$\phi_{z_1,\dots,z_n}(\partial_t^n S_l)$. The latter are Taylor
coefficients of
$S_l(w;z_1,\dots,z_n)=\phi_{w-z_1,\dots,w-z_n}(S_l)$ about $w=0$.
The function $S_l(w;z_1,\dots,z_n)$ is meromorphic in $w$ having a
zero of order $\deg\Phi_l$ at $\infty$. Since the poles of
$S_l(w)$ are exactly $z_1,\dots,z_n$, the Taylor coefficients of
$S_l(w)$ about $w=0$ are linear expressions in the coefficients of
the principal part of the Laurent series for the same function
about $z_1,\dots,z_n$. Since the function $S_l(w)$ has zero of
order $\deg\Phi_l$ at $\infty$, for each $m=1,\dots,\deg\Phi_l-1$
the Laurent coefficient $S_l^{n,m}(z_1,\dots,z_n)=\Res_{w=z_n}
(w-z_n)^{m-1}S_l(w;z_1,\dots,z_n)$ is a linear combination of the
Laurent coefficients at $z_1,\dots,z_{n-1}$.

Now, it remains to check that the generators $S_l^{i,m}$ are
algebraically independent. Equivalently, we need to prove that the
transcendence degree of $\A(z_1,\dots,z_n)$ is
$\frac{n-1}{2}(\dim\gg+\rk\gg)+\rk\gg$. We shall deduce this from
the maximality of quantum Mishchenko-Fomenko subalgebras. Namely,
it is sufficient to prove that the associated graded quotient of
$\A(z_1,\dots,z_n)$ with respect to the $n$-th tensor factor has
the transcendence degree $\frac{n-1}{2}(\dim\gg+\rk\gg)+\rk\gg$.
This associated graded quotient is a commutative subalgebra
$\overline{\A(z_1,\dots,z_n)}\subset U(\gg)^{\otimes(n-1)}\otimes
S(\gg)$. Any element $\mu\in\gg^*=\Spec S(\gg)$ gives the
evaluation homomorphism
$\id^{\otimes(n-1)}\otimes\mu:U(\gg)^{\otimes(n-1)}\otimes
S(\gg)$. This homomorphism sends the central generators
$\overline{S_l^{n,\deg\Phi_l}(z_1,\dots,z_n)}=\Phi_l^{(n)}\in
\overline{\A(z_1,\dots,z_n)}\subset U(\gg)^{\otimes(n-1)}\otimes
S(\gg)$ to constants. By Theorems~2~and~3 of \cite{Ryb1}, the
algebra
$(\id^{\otimes(n-1)}\otimes\mu)(\overline{\A(z_1,\dots,z_n)})\subset
U(\gg)^{\otimes(n-1)}$ is the tensor product of quantum
Mishchenko-Fomenko subalgebras,
$\A_{\mu}^{(1)}\otimes\dots\otimes\A_{\mu}^{(n-1)}$ which is known
to be of the transcendence degree $\frac{n-1}{2}(\dim\gg+\rk\gg)$
(see \cite{MF}). Thus, the algebra $\overline{\A(z_1,\dots,z_n)}$
has the transcendence degree $\frac{n-1}{2}(\dim\gg+\rk\gg)$ over
$\cc[\Phi_1^{(n)},\dots,\Phi_{\rk\gg}^{(n)}]=S(\gg)^\gg$. Hence
the assertion.

(2) The Laurent coefficients of the functions
$S_l(w;z_1+b,\dots,z_n+b)$ about $z_i+b$ are equal to those of
$S_l(w-b;z_1,\dots,z_n)$ about $z_i$, hence our generators are
stable under simultaneous transformations of the parameters
$z_i\mapsto z_i+b$.

Now consider simultaneous transformations $z_i\mapsto az_i$. The
Laurent coefficients of the functions $S_l(w;az_1,\dots,az_n)$
about $az_i$ are proportional to those of
$S_l(aw;az_1,\dots,az_n)$ about $z_i$. Since the elements $S_l\in
U(\gg_-)$ are homogeneous with respect to $t\partial_t$, we see
that $\exp((\log a)t\partial_t)S_l$ is proportional to $S_l$. Note
that $S_l(aw;az_1,\dots,az_n)=\phi_{w-z_1,\dots,w-z_n}(\exp((\log
a)t\partial_t)S_l)$. This means that the Laurent coefficients of
the functions $S_l(aw;az_1,\dots,az_n)$ about $z_i$ are
proportional to those of $S_l(w;z_1,\dots,z_n)$ about $z_i$.

(3) The elements $S_l\in U(\gg_-)$ are $\gg$-invariant, and the
homomorphisms $\phi_{w-z_1,\dots,w-z_n}$ are $\gg$-equivariant for
any $w$, hence the images of $S_l$ under the homomorphisms
$\phi_{w-z_1,\dots,w-z_n}$ are $\gg$-invariant as well.

(4) The $\deg\Phi_l$-th Taylor coefficient of
$S_l(w;z_1,\dots,z_n)$ at $\infty$ is the $l$-th generator of the
center of $\diag_n(U(\gg))\subset U(\gg)^{\otimes n}$.
\end{proof}

\section{Limits of Gaudin algebras}\label{sect_limit}

The algebra $U(\gg)^{\otimes n}$ has an increasing filtration by
finite-dimensional spaces, $U(\gg)^{\otimes
n}=\bigcup\limits_{k=0}^{\infty} (U(\gg)^{\otimes n})_{(k)}$ (by
degree with respect to the generators). We define the limit
$\lim\limits_{s\to\infty}B(s)$ for any one-parameter family of
subalgebras $B(s)\subset U(\gg)^{\otimes n}$ as
$\bigcup\limits_{k=0}^{\infty} \lim\limits_{s\to\infty}B(s)\cap
(U(\gg)^{\otimes n})_{(k)}$. It is clear that the limit of a
family of \emph{commutative} subalgebras is a commutative
subalgebra. It is also clear that passage to the limit commutes
with homomorphisms of filtered algebras (in particular, with the
projection onto any factor and with finite-dimensional
representations).

We shall consider the limits of Gaudin subalgebras when some of
the points $z_1,\dots,z_n$ glue together. More precisely, let
$z_1,\dots,z_k$ be independent on $s$, and $z_{k+i}=z+su_i,\
i=1,\dots,n-k$, where $z_1,\dots,z_k,z\in\cc$ are pairwise
distinct and $u_1,\dots,u_{n-k}\in\cc$ are pairwise distinct. Let
us describe the limit subalgebra
$\lim\limits_{s\to0}\A(z_1,\dots,z_k,z+su_1,\dots,z+su_{n-k})\subset
U(\gg)^{\otimes n}$.

Consider the following homomorphisms $$D_{k,n}:=\id^{\otimes
k}\otimes\diag_{n-k}:U(\gg)^{\otimes (k+1)}\hookrightarrow
U(\gg)^{\otimes n},$$ and
$$I_{k,n}:=1^{\otimes k}\otimes\id^{\otimes(n-k)}:
U(\gg)^{\otimes (n-k)}\hookrightarrow U(\gg)^{\otimes n},
$$
where $\id:U(\gg)\to U(\gg)$ is the 
identity map, $\diag_{n-k}:U(\gg)\hookrightarrow U(\gg)^{\otimes
(n-k)}$ is the diagonal embedding.
Clearly,
the image of $[U(\gg)^{\otimes (n-k)}]^\gg$
under the homomorphism
$I_{k,n}$ commutes with the image of the homomorphism $D_{k,n}$.

Let $z_1,\dots,z_k,z$ and $u_1,\dots,u_{n-k}$ be two collections
of pairwise distinct complex numbers. We assign to these data a
commutative subalgebra
$$
\A_{(z_1,\dots,z_k,z),(u_1,\dots,u_{n-k})}:=
D_{k,n}(\A(z_1,\dots,z_k,z))\cdot I_{k,n}(\A(u_1,\dots,u_{n-k}))
\subset U(\gg)^{\otimes n}.
$$

\begin{Proposition}\label{tr-deg}
The subalgebra $\A_{(z_1,\dots,z_k,z),(u_1,\dots,u_{n-k})}$ is a
free commutative algebra generated by the elements
$D_{k,n}(S_l^{i,m}(z_1,\dots,z_k,z))$, with $i=1,\dots,k$,
$l=1,\dots,\rk\gg$, $m=1,\dots,\deg\Phi_l$,
$I_{k,n}(S_l^{i,m}(u_1,\dots,u_{n-k}))$, with $i=1,\dots,n-k-1$,
$l=1,\dots,\rk\gg$, $m=1,\dots,\deg\Phi_l$ and
$S_l^{n,\deg\Phi_l}(z_1,\dots,z_n)=\Phi_l^{(n)}\in U(\gg)^{\otimes
n}$, where $l=1,\dots,\rk\gg$.
\end{Proposition}

\begin{proof} Note that by the assertion (4) of
Proposition~\ref{generators2} the center of $\id^{\otimes
k}\otimes\diag_{n-k}(U(\gg))$ is contained in
$I_{k,n}(\A(u_1,\dots,u_{n-k})$. Hence, by (1) of
Proposition~\ref{generators2}, the elements defined above generate
the algebra $\A_{(z_1,\dots,z_k,z),(u_1,\dots,u_{n-k})}$. Thus it
remains to show that these elements are algebraically independent.
But this is so due to the same argument as (1) of
Proposition~\ref{generators2}.
\end{proof}

\begin{Theorem}
$\lim\limits_{s\to0}\A(z_1,\dots,z_k,z+su_1,\dots,z+su_{n-k})=\A_{(z_1,\dots,z_k,z),(u_1,\dots,u_{n-k})}\subset
U(\gg)^{\otimes n}$.
\end{Theorem}
\begin{proof}
\begin{Lemma}\label{predel}
$\lim\limits_{z\to\infty}\phi_z=\ep$, where
$\ep:U(\gg_-)\to\cc\cdot1\subset~U(\gg)$ is the co-unit.
\end{Lemma}
\begin{proof}
It is sufficient to check this on the generators. We have
$$
\lim\limits_{z\to\infty}\phi_z(g[m])=\lim\limits_{z\to\infty}z^mg=0\quad\forall\
g\in\gg,\ m=-1,-2,\dots.
$$
\end{proof}

Let us choose the generators of
$\A(z_1,\dots,z_k,z+su_1,\dots,z+su_{n-k})$ as in
Proposition~\ref{generators2}. The coefficients of the Laurent
expansion of $S_l(w;z_1,\dots,z_k,z+su_1,\dots,z+su_{n-k})$ at
$z+su_j$ are proportional to the Laurent coefficients of
$S_l(w;\frac{z_1-z}{s},\dots,\frac{z_k-z}{s},u_1,\dots,u_{n-k})$
at the point $u_j$. On the other hand, by Lemma~\ref{predel} we
have
\begin{multline*}
\lim\limits_{s\to0}S_l(w;\frac{z_1-z}{s},\dots,\frac{z_k-z}{s},u_1,\dots,u_{n-k})=\\=
\lim\limits_{s\to0}\phi_{w-\frac{z_1-z}{s},\dots,w-\frac{z_k-z}{s},w-u_1\dots,w-u_{n-k}}(S_l)=\\
=(\ep\otimes\dots\otimes\ep\otimes\phi_{w-u_1}\otimes\dots\otimes\phi_{w-u_{n-k}})\circ
\diag_n(S_l)=I_{n,k}S_l(w;u_1,\dots,u_{n-k}).
\end{multline*}

Therefore,
$\lim\limits_{s\to0}S_l^{i+k,m}(z_1,\dots,z_k,z+su_1,\dots,z+su_{n-k})=I_{n,k}(S_l^{i,m}(u_1,\dots,u_{n-k}))$
for $i=1,\dots,n-k$.

Now let us compute the limits of the coefficients of the Laurent
expansion of $S_l(w;z_1,\dots,z_k,z+su_1,\dots,z+su_{n-k})$ at
$z_i$.

\begin{Lemma}\label{predel2}
$\lim\limits_{s\to0}\phi_{z_1,\dots,z_k,z+su_1,\dots,z+su_{n-k}}=D_{n,k}\circ\phi_{z_1,\dots,z_k,z}$.
\end{Lemma}
\begin{proof}
It is sufficient to check this on the generators. We have
\begin{multline*}
\lim\limits_{s\to0}\phi_{z_1,\dots,z_k,z+su_1,\dots,z+su_{n-k}}(g[m])=\\=
\lim\limits_{s\to0}(\sum\limits_{i=1}^kz_i^mg^{(i)}+\sum\limits_{i=k+1}^n(z+su_{i-k})^mg^{(i)})=\\=
D_{n,k}\circ\phi_{z_1,\dots,z_k,z}(g[m]).
\end{multline*}
\end{proof}

By Lemma~\ref{predel2} we have
$\lim\limits_{s\to0}S_l^{i,m}(z_1,\dots,z_k,z+su_1,\dots,z+su_{n-k})=D_{n,k}(S_l^{i,m}(z_1,\dots,z_k,z))$
for $i=1,\dots,k$.

Thus we have
$\lim\limits_{s\to0}\A(z_1,\dots,z_k,z+su_1,\dots,z+su_{n-k})\supset\A_{(z_1,\dots,z_k,z),(u_1,\dots,u_{n-k})}$.
On the other hand, from
propositions~\ref{generators2}~and~\ref{tr-deg}, it follows that
the algebras $\A(z_1,\dots,z_k,z+su_1,\dots,z+su_{n-k})$ and of
$\A_{(z_1,\dots,z_k,z),(u_1,\dots,u_{n-k})}$ have the same
Poincare series. Hence
$\lim\limits_{s\to0}\A(z_1,\dots,z_k,z+su_1,\dots,z+su_{n-k})=\A_{(z_1,\dots,z_k,z),(u_1,\dots,u_{n-k})}$
\end{proof}

\section{Quantum hamiltonians of bending flows}\label{sect_bendflows}

In \cite{FM03-1}, a complete system of Poisson-commuting elements in
the Poisson algebra $S(gl_N)^{\otimes n}$ was constructed. Namely,
the following functions on $gl_n\oplus\dots\oplus gl_N$ commute
with respect to the Poisson bracket.
$$\overline H_{l,k}^{(\alpha)}(X_1,\dots,X_n):=\Res_{s=0}\frac{1}{z^{\alpha+1}}\Tr(X_k+z(\sum\limits_{i=k+1}^nX_i))^l,$$
where $X_1,\dots,X_n\in gl_N$. Now, we construct a system of
commuting elements $H_{l,k}^{(\alpha)}\in U(gl_N)^{\otimes n}$
such that $\gr H_{l,k}^{(\alpha)}=\overline H_{l,k}^{(\alpha)}$.

One can iterate the limiting procedure described in the previous
section to obtain some new commutative subalgebras in
$U(\gg)^{\otimes n}$. In particular, we can obtain the following
subalgebra $\A_{(z_1,z_2),\dots,(z_1,z_2)}\subset U(\gg)^{\otimes
n}$, which is generated by
$$D_{n,2}(\A_{z_1,z_2}),\ 1\otimes D_{n-1,2}(\A_{z_1,z_2}),\
\dots,\ 1^{\otimes(n-2)}\otimes \A_{z_1,z_2}.$$ By
Proposition~\ref{generators2} assertion (2), this subalgebra is
independent on $z_1,z_2$. We denote it by $\A_{\lim}$. The
following is checked by an easy computation.

\begin{Proposition} For $\Phi_l(X)=\Tr X^l\ (X\in gl_N)$, $S_l(z)=i_z(\Phi_l)$,
we have $\gr D_{n-k+1,2}(S_l^{1,\alpha})=\overline
H_{l,k}^{(\alpha)}$.
\end{Proposition}

This means that, for $\gg=gl_N$, the algebra $\A_{\lim}$ gives a
quantization of the "bending flows" system from \cite{FM03-1}.

\section{Schur--Weyl duality and Jucys--Murphy elements}\label{sect_jucis}

Let $\gg=gl_N$, and $V_{(\l)}=\cc^N\otimes\dots\otimes\cc^N$. By
Schur--Weyl duality, the centralizer of the diagonal $\gg$ in
$\End(V_{(\l)})$ is the image of the group algebra $\cc[S_n]$.
Equivalently, the space
$V_{(\l)}^{sing}=[\cc^N\otimes\dots\otimes\cc^N]^{sing}$
decomposes into the sum of irreducible $S_n$-modules with
multiplicities $0$ or $1$. Since the elements of
$\A(z_1,\dots,z_n)$ commute with the diagonal action of $\gg$, we
can treat them as commuting elements of $\cc[S_n]$. In particular,
one can rewrite quadratic Gaudin hamiltonians~\ref{quadratic} as
follows:
$$
H_i=\sum\limits_{j\ne i} \frac{(i,j)}{z_i-z_j}.
$$

The quadratic elements of $\A_{\lim}$ can be rewritten as follows
$$
H_i=\sum\limits_{j<i} (i,j).
$$
The latter are known as Jucys--Murphy elements. By \cite{OV},
these elements generate the Gelfand--Tsetlin subalgebra in
$\cc[S_n]$ (in other words, this subalgebra is generated by the
centers of the group subalgebra $\cc[S_{n-1}]\subset\cc[S_n]$,
$\cc[S_{n-2}]\subset\cc[S_{n-1}]\subset\cc[S_n]$ and so on). This
algebra has a simple spectrum in any irreducible representation of
$S_n$. We can obtain from this the following result of Mukhin and
Varchenko.

\begin{Proposition}\cite{MV} Suppose $\gg=gl_N$ and $(\l)=(\l_1,\dots,\l_n)$, where $\l_i=\omega_1$.
The Gaudin algebra $\A(z_1,\dots,z_n)$ (and, moreover, its
quadratic part) has simple joint spectrum in $V_{(\l)}^{sing}$ for
generic values of the parameters $z_1,\dots,z_n$.
\end{Proposition}

\begin{proof}
The Gelfand--Tsetlin subalgebra in $\cc[S_n]$ has a simple
spectrum in any irreducible representation of $S_n$. This means
that the algebra $\A_{\lim}$ has a simple spectrum in
$V_{(\l)}^{sing}=[\cc^N\otimes\dots\otimes\cc^N]^{sing}$ (since
the latter is multiplicity-free as an $S_n$-module). Since the
subalgebra $\A_{\lim}$ belongs to the closure of the family of
Gaudin subalgebras $\A(z_1,\dots,z_n)$, for generic values of
$z_i$ the algebra $\A(z_1,\dots,z_n)$ has simple spectrum in
$V_{(\l)}^{sing}$ as well.
\end{proof}

\section{$(gl_N,\ gl_M)$ duality and Gelfand--Tsetlin algebra}\label{sect_duality}

Consider the space $W:=\cc^N\otimes\cc^M$ with the natural action
of the Lie algebra $gl_N\oplus gl_M$. The universal enveloping
algebra $U(gl_N\oplus gl_M)=U(gl_N)\otimes U(gl_M)$ acts on the
symmetric algebra $S(W)=\cc[W^*]$ by differential operators. Let
$D(W)$ be the algebra of differential operators on $W^*$. We shall
use the following classical result.

\begin{Fact}\label{gl-gl-duality}
The image of $U(gl_M)$ is the centralizer of the image of
$U(gl_N)$ in $D(W)$. Equivalently, the space $S(W)^{sing}$ is
multiplicity-free as $gl_M$-module.
\end{Fact}

We can treat the space $W$ as the direct sum of $M$ copies of
$\cc^N$, and hence we have the action of $U(gl_N\oplus\dots\oplus
gl_N)=U(gl_N)^{\otimes M}$ on $S(W)$. The elements of the Gaudin
subalgebra $\A(z_1,\dots,z_M)\subset U(gl_N)^{\otimes M}$ commute
with the diagonal $gl_N$, and hence they can be rewritten as the
elements of $U(gl_M)$. Thus, we obtain a commutative subalgebra in
$U(gl_M)$. Let us describe this subalgebra explicitly. Consider
the diagonal $M\times M$-matrix $Z$ with the diagonal entries
equal to $z_1,\dots,z_M$. To any diagonal matrix $Z$ one can
naturally assign a quantum Mishchenko-Fomenko subalgebra
$\A_Z\subset U(gl_M)$. For generic $Z$, the subalgebra
$\A_Z\subset U(gl_M)$ is the centralizer of the following family
of commuting quadratic elements \cite{Ryb2}:
\begin{equation}\label{quadratic_MF}
Q_Z:=\{\sum\limits_{\alpha\in\Delta_+}\frac{\langle\alpha,h\rangle}{\langle\alpha,Z\rangle}e_{\alpha}e_{-\alpha}|h\in\hh\},
\end{equation}
where $\hh\subset gl_M$ is the subalgebra of diagonal matrices,
$\Delta$ is the root system of $gl_M$, $\Delta_+$ is the set of
positive roots, and $e_{\alpha}$ are certain nonzero elements of
the root spaces $\gg_{\alpha},\ \alpha\in\Delta$. The following
assertion is due to Mukhin, Tarasov and Varchenko.

\begin{Fact}
\cite{MTV2} The image of the Gaudin subalgebra
$\A(z_1,\dots,z_M)\subset U(gl_N)^{\otimes M}$ in $\End S(W)$
coincides with the image of $\A_Z\subset U(gl_M)$. The space of
quadratic Gaudin hamiltonians coincides with the image of
$Q_Z\subset U(gl_M)$.
\end{Fact}

{\bf Remark.} The quasi-classical version of this fact goes back
to \cite{AHH90} (see also \cite{GGM97} section 5.4 and especially
formula 5.27 page 23). But the relation with the $(gl_N,\ gl_M)$
duality was not understood and the full picture was not developed.
These facts go back to the well-known in integrability theory fact
that the Toda chain has two Lax representations: one is by
$2\times 2$ matrices another by $n\times n$ matrices.

{\bf Remark.} The $(gl_N,\ gl_M)$ duality for quadratic
hamiltonians was first observed by Toledano Laredo in \cite{TL}.

In \cite{Sh} Shuvalov described the closure of the family of
subalgebras $A_Z\subset S(\gg)$ under the condition
$Z\in\hh^{reg}$ (i.e., for regular $Z$ in the fixed Cartan
subalgebra). In particular, the following assertion is proved
in~\cite{Sh}.

\begin{Fact}\label{shuvalov} Suppose that $Z(t)=Z_0+tZ_1+t^2Z_2+\dots\in\hh^{reg}$
for generic $t$. Set $\z_k=\bigcap\limits_{i=0}^k\z_{\gg}(\mu_i)$
(where $\z_{\gg}(Z_i)$ is the centralizer of $Z_i$ in $\gg$),
$\z_{-1}=\gg$. Then we have
\begin{enumerate}
\item the subalgebra $\lim\limits_{t\to0}A_{Z(t)}\subset S(\gg)$
is generated by all elements of $S(\z_k)^{\z_k}$ and their
derivatives (of any order) along $Z_{k+1}$ for all $k$. \item
$\lim\limits_{t\to0}A_{Z(t)}$ is a free commutative algebra.
\end{enumerate}
\end{Fact}

This means, in particular, that the $\lim\limits_{z_1\to
z_M}(\dots(\lim\limits_{z_{M-2}\to z_M}(\lim\limits_{z_{M-1}\to
z_M}A_{Z}))\dots)$ for $\gg=gl_M$ is generated by the Poisson
centers of $S(gl_{M-1})\subset S(gl_M)$, of $S(gl_{M-2})\subset
S(gl_{M-1})\subset S(gl_M)$ etc. (this is a maximal commutative
subalgebra known as the Gelfand--Tsetlin algebra). This was
observed earlier by Vinberg in \cite{Vin}, 6.1--6.4.

Let $\A_{G-Ts}\subset U(gl_M)$ be the Gelfand--Tsetlin commutative
subalgebra generated by the center of $U(gl_{M-1})\subset
U(gl_M)$, the center of $U(gl_{M-2})\subset U(gl_{M-1})\subset
U(gl_M)$ etc.

\begin{Theorem} The image of $\A_{\lim}\subset
U(gl_N)^{\otimes M}$ in $\End S(W)$ coincides with
$\A_{G-Ts}\subset U(gl_M)$.
\end{Theorem}

\begin{proof}
It is sufficient to prove, that the corresponding limit of
$\A_Z\subset U(gl_M)$ is the Gelfand-Tsetlin subalgebra. On the
level of Poisson algebras this is proved in \cite{Vin} and
\cite{Sh}. On the other hand, in \cite{Tar3}, it is proved that
there is a unique lifting of such limit subalgebra to $U(gl_M)$.
Thus, the limit of $\A_Z\subset U(gl_M)$ is $\A_{G-Ts}\subset
U(gl_M)$.
\end{proof}

\section{Simplicity of the joint spectrum of the Gaudin algebras in the $gl_N$ case}\label{sect_simplespec}
Let $V$ be the tautological representation of $\gg=gl_N$ (of the
highest weight $\omega_1$) and
$V_{(\l)}=\bigotimes\limits_{i=1}^{n} S^{m_i}V$ (or, equivalently,
$(\l)=(\l_1,\dots,\l_n)$, where $\l_i=m_i\omega_1,\ m_i\in\zz_+$).
We shall show in this section, that the joint spectrum of the
Gaudin algebra $\A(z_1,\dots,z_n)$ in the space $V_{(\l)}^{sing}$
is simple for generic values of the parameters $z_1,\dots,z_n$.

\begin{Theorem}\label{simple_spec} Suppose $(\l)=(\l_1,\dots,\l_n)$, where $\l_i=m_i\omega_1,\
m_i\in\zz_+$.
The Gaudin algebra $\A(z_1,\dots,z_n)$ has
simple spectrum in $V_{(\l)}^{sing}$ for generic values of the
parameters $z_1,\dots,z_n$.
\end{Theorem}

\begin{proof} The Gelfand--Tsetlin subalgebra in $U(gl_M)$ has
simple spectrum in any irreducible representation of $gl_M$. Since
the image of $\A_{G-Ts}\subset U(gl_M)$ coincides with
$\A_{\lim}$, and $S(W)^{sing}$ is multiplicity-free as
$gl_M$-module, the algebra $\A_{\lim}$ has a simple spectrum in
$S(W)^{sing}$.

Note that the representation $V_{(\l)}$ occurs in $S(W)$. Thus the
algebra $\A_{\lim}$ has simple spectrum in $V_{(\l)}^{sing}$.

Since the subalgebra $\A_{\lim}$ belongs to the closure of the
family of Gaudin subalgebras $\A(z_1,\dots,z_n)$, for generic
values of $z_i$ the algebra $\A(z_1,\dots,z_n)$ has simple
spectrum in $V_{(\l)}$ as well.
\end{proof}


\begin{thebibliography}{10}
\bibitem[AHH90]{AHH90}
MR Adams, J Harnad, J Hurtubise
{\em Dual moment maps into loop algebras},
Lett.Math.Phys. 20,(1990),299

\bibitem[BR98]{Rag}
A. Ballesteros, O. Ragnisco, {\em A systematic construction
of completely integrable Hamiltonians from coalgebras},
J. Phys. A {\bf 31} (1998), 3791--3813


\bibitem[BD]{BD} A. Beilinson and V. Drinfeld, {\em Quantization of Hitchin's
integrable system and Hecke eigen-sheaves}, Preprint, available at
\href{http://www.ma.utexas.edu/~benzvi/BD}{www.ma.utexas.edu/$\sim$benzvi/BD}.

\bibitem[BF]{BF} H.M. Babujian and R. Flume, {\em Off-shell Bethe Ansatz
equation for Gaudin magnets and solutions of
Knizhnik-Zamolodchikov equations}, Mod. Phys. Lett. {\bf A 9}
(1994) 2029--2039.

\bibitem[CF07]{CF}
A. Chervov, G. Falqui, {\em Manin's matrices and Talalaev's
formula},
\href{http://lanl.arxiv.org/abs/0711.2236}{arXiv:0711.2236}


\bibitem[CT04]{CT04} Chervov, A. and Talalaev, D.
\emph{Universal G-oper and Gaudin eigenproblem},
\href{http://arxiv.org/abs/hep-th/0409007}{hep-th/0409007}

\bibitem[CT06]{CT06} Chervov, A. and Talalaev, D.
\emph{Quantum spectral curves, quantum integrable systems and the
geometric Langlands correspondence},
\href{http://xxx.lanl.gov/abs/hep-th/0604128}{hep-th/0604128}

\bibitem[CRT]{CRT} Chervov, A., Rybnikov L., and Talalaev, D.
\emph{Rational Lax operators and their quantization},
\href{http://arxiv.org/abs/hep-th/0404106}{hep-th/0404106}

\bibitem[ER]{ER} B.Enriquez, V.Rubtsov, {\em Hitchin systems, higher Gaudin hamiltonians and $r$-matrices.}
Math. Res. Lett.  3  (1996), no. 3, 343--357.
\href{http://arxiv.org/abs/alg-geom/9503010}{alg-geom/9503010}

\bibitem[FlM01]{FlashM01}
H. Flaschka, J. Millson,
{\em The moduli space of weighted
configurations on projective space},
\href{http://arxiv.org/abs/math.SG/0108191}{math.SG/0108191}

\bibitem[FM03-1]{FM03-1}
Gregorio Falqui, Fabio Musso,
{\em Gaudin Models and Bending Flows: a Geometrical Point of View},
J. Phys. A  36  (2003), no. 46,11655--11676.
\href{http://arxiv.org/abs/nlin.SI/0306005}{nlin.SI/0306005}

\bibitem[FM03-2]{FM03-2}
Gregorio Falqui, Fabio Musso,
{\em Bi-hamiltonian Geometry and Separation of Variables for Gaudin Models: a case study},
\href{http://arxiv.org/abs/nlin.SI/0306008}{nlin.SI/0306008}

\bibitem[FM04]{FM04}
Gregorio Falqui, Fabio Musso,
{\em On Separation of Variables for Homogeneous SL(r) Gaudin Systems},
\href{http://arxiv.org/abs/nlin.SI/0402026}{nlin.SI/0402026}

\bibitem[FM06]{FM06}
Gregorio Falqui, Fabio Musso,
{\em Quantisation of bending flows},
\href{http://arxiv.org/abs/nlin.SI/0610003}{nlin.SI/0610003}


\bibitem[GGM97]{GGM97}
A.Gorsky, S.Gukov, A.Mironov,
{\em Multiscale N=2 SUSY field theories, integrable systems and their stringy/brane origin -- I}
\href{http://arxiv.org/abs/hep-th/9707120}{hep-th/9707120}
%

\bibitem[FF]{FF} B. Feigin and E. Frenkel, {\em Affine Kac--Moody
algebras at the critical level and Gelfand--Dikii algebras}, Int. Jour.
Mod. Phys. {\bf A7}, Supplement 1A (1992) 197--215.

\bibitem[FFR]{FFR} B.Feigin, E.Frenkel, N.Reshetikhin, {\em Gaudin model, Bethe Ansatz and critical level.}
Comm. Math. Phys., 166 (1994), pp. 27-62.

\bibitem[FFTL]{FFTL} B.Feigin, E.Frenkel, V. Toledano Laredo, {\em Gaudin model with irregular singularities.}
\href{http://arxiv.org/abs/math.QA/0612798}{math.QA/0612798}.

\bibitem[Fr95]{Fr1}
 E.Frenkel,
{\em Affine Algebras, Langlands Duality and Bethe Ansatz}, XIth
International Congress of Mathematical Physics (Paris, 1994),
606--642, Internat. Press, Cambridge, MA, 1995.
\href{http://arxiv.org/abs/q-alg/9506003}{q-alg/9506003}

\bibitem[Fr02]{Fr2}
 E.Frenkel,
{\em Lectures on Wakimoto modules, opers and the center at the
critical level},
\href{http://arxiv.org/abs/math.QA/0210029}{math.QA/0210029}.

\bibitem[G76]{G1}
M.Gaudin, {\em Diagonalisation d'une classe d'hamiltoniens de
spin}, J. de Physique, t.37, N 10, p. 1087--1098, 1976.

\bibitem[G83]{G}Gaudin, Michel. \emph{La fonction d'onde de Bethe.} (French)
[The Bethe wave function] Collection du Commissariat a` l'E'nergie
Atomique: Se'rie Scientifique. [Collection of the Atomic Energy
Commission: Science Series] Masson, Paris, 1983. xvi+331 pp.

\bibitem[KM96]{KM96}
M. Kapovich, J. Millson,  {\em The symplectic geometry
of polygons in Euclidean space},J. Differ. Geom. {\bf{44}}, 479--513 (1996)

\bibitem[MF]{MF}
Mishchenko, A. S. and Fomenko, A. T. \emph{Integrability of
Euler's equations on semisimple Lie algebras.} (Russian) Trudy
Sem. Vektor. Tenzor. Anal. No. 19 (1979), 3--94.


\bibitem[MTV05]{MTV} E. Mukhin, V. Tarasov and A. Varchenko, {\em Bispectral and ($gl_N,\ gl_M$) dualities,}
\href{http://arxiv.org/abs/math.QA/0510364}{math.QA/0510364}.

\bibitem[MTV0512]{MTVShap}
E. Mukhin, V. Tarasov, A. Varchenko, {\em The B. and M. Shapiro
conjecture in real algebraic geometry and the Bethe ansatz},
\href{http://lanl.arxiv.org/abs/math/0512299}{arXiv:math/0512299}

\bibitem[MTV06-1]{MTV2} E. Mukhin, V. Tarasov and A. Varchenko, {\em A generalization of the Capelli identity,}
\href{http://arxiv.org/abs/math.QA/0610799}{math.QA/0610799}.

\bibitem[MTV06-2]{MTV06} E. Mukhin, V. Tarasov and A. Varchenko, {\em Bethe eigenvectors of higher transfer matrices,}
\href{http://arxiv.org/abs/math.QA/0605015}{math.QA/0605015}.

\bibitem[MV]{MV} E. Mukhin and A. Varchenko, {\em Norm of a Bethe vector and the Hessian
of the master function}, Compos. Math.  141  (2005), no. 4,
1012--1028. \href{http://arxiv.org/abs/math.QA/0402349}{math.QA/0402349}.



\bibitem[OV]{OV} Okounkov, Andrei and Vershik, Anatoly. \emph{A new approach to
representation theory of symmetric groups.} Selecta Math. (N.S.) 2
(1996), no. 4, 581--605.
\href{http://arxiv.org/abs/math.RT/0503040}{math.RT/0503040}.

\bibitem[Ryb]{Ryb1} Rybnikov, L.~G., {\em Argument shift method
and Gaudin model}, Func. Anal. Appl., {\bf 40} (2006), No 3,
translated from Funktsional'nyi Analiz i Ego Prilozheniya, vol. 40
(2006), No~3, pp.~30--43.
\href{http://arxiv.org/abs/math.RT/0606380}{math.RT/0606380}.

\bibitem[Ryb05]{Ryb2} L.~G.~Rybnikov, \emph{Centralizers of certain quadratic elements
in Poisson-Lie algebras and Argument Shift method. (Russian)}
Uspekhi Mat. Nauk 60 (2005), no. 2(362), 173--174;
\href{http://arxiv.org/abs/math.QA/0608586}{math.QA/0608586}.

\bibitem[Ryb06]{Ryb}  L.G.Rybnikov, \emph{Uniqueness
of higher Gaudin hamiltonians,}
\href{http://arxiv.org/abs/math.QA/0608588}{math.QA/0608588}.

\bibitem[Sh]{Sh} Shuvalov, V. V. \emph{On the limits of Mishchenko-Fomenko
subalgebras in Poisson algebras of semisimple Lie algebras.}
Russian) Funktsional. Anal. i Prilozhen.  36  (2002), no. 4,
55--64; translation in  Funct. Anal. Appl.  36  (2002), no. 4,
298--305

\bibitem[SV]{SV}Scherbak, I.; Varchenko, A. \emph{Critical points of functions, $sl_2$
representations, and Fuchsian differential equations with only
univalued solutions.} Dedicated to Vladimir I. Arnold on the
occasion of his 65th birthday. Mosc. Math. J. 3 (2003), no. 2,
621--645, 745.


\bibitem[T04]{T04} Talalaev, D. {\em Quantization of the Gaudin system},
Functional Analysis and Its application Vol. {\bf 40} No. 1 pp.86-91 (2006)
\href{http://arxiv.org/abs/hep-th/0404153}{hep-th/0404153}


\bibitem[Tar]{Tar3}
Tarasov, A. A. \emph{On the uniqueness of the lifting of maximal
commutative subalgebras of the Poisson-Lie algebra to the
enveloping algebra.} (Russian)  Mat. Sb.  194  (2003), no. 7,
155--160;  translation in  Sb. Math.  194  (2003), no. 7-8,
1105--1111.

\bibitem[TL]{TL}
Toledano Laredo, Valerio. \emph{A Kohno-Drinfeld theorem for
quantum Weyl groups.} Duke Math. J.  112  (2002), no. 3, 421--451,
\href{http://arxiv.org/abs/math/0009181}{math.QA/0009181}.

\bibitem[Vin]{Vin}
E. Vinberg, {\em On some commutative subalgebras in universal
enveloping algebra}, Izv. AN USSR, Ser. Mat., 1990, vol. 54 N 1 pp
3-25



\end{thebibliography}
\end{document}